\documentclass{article}
\usepackage{amsmath}
\usepackage{amssymb}
\usepackage{amsthm,bbm}
\usepackage{hyperref}

\newcommand{\midb}{\;\middle|\;}
\newcommand{\one}{\mathbbm 1}

\def\reals{\mathbb{R}}

\def\ereals{\overline{\mathbb{R}}}

\def\comp{\raise 1pt \hbox{$\scriptstyle\circ$}}

\def\minimize{\mathop{\rm minimize}\limits}

\def\essinf{\mathop{\rm ess\ inf}\nolimits}

\def\dom{\mathop{\rm dom}\nolimits}

\def\upto{{\raise 1pt \hbox{$\scriptstyle \,\nearrow\,$}}}
\def\downto{{\raise 1pt \hbox{$\scriptstyle \,\searrow\,$}}}

\def\co{\mathop{\rm co}}

\def\FF{(\F_t)_{t=0}^T}
\def\one{\mathbbm 1}
\def\ovr{\mathop{\rm over}}

\def\B{{\cal B}}
\def\C{{\cal C}}

\def\F{{\cal F}}

\def\L{{\cal L}}

\def\N{{\cal N}}

\def\R{{\mathbb R}}

\def\Q{{\cal Q}}

\def\U{{\cal U}}

\def\Y{{\cal Y}}

\newtheorem{theorem}{Theorem}
\newtheorem{lemma}[theorem]{Lemma}
\newtheorem{corollary}[theorem]{Corollary}

\newtheorem{example}{Example}
\newtheorem{remark}{Remark}
\theoremstyle{definition}

\begin{document}

\title{Stochastic programs without duality gaps for objectives without a lower bound}

\author{Ari-Pekka Perkki\"o\footnote{Department of Mathematics, Technische Universit\"at Berlin, \href{mailto:perkkioe@math.tu-berlin.de}{perkkioe@math.tu-berlin.de}}}

\maketitle

\begin{abstract}
This paper studies parameterized stochastic optimization problems in finite discrete time that arise in many applications in operations research and mathematical finance. We prove the existence of solutions and the absence of a duality gap under conditions that relax the boundedness assumption made by Pennanen and Perkki\"o in [Stochastic programs without duality gaps, Math. Program., 136(1):91--110,2012]. We apply the result to a utility maximization problem with an unbounded utility.
\end{abstract}
\section{Introduction}

Let $(\Omega,\F,P)$ be a complete probability space with a filtration $\FF$ of complete sub sigma-algebras of $\F$ and consider the parametric dynamic stochastic optimization problem
\begin{equation}\label{p}\tag{$P_u$}
\minimize\quad Ef(x,u):=\int f(x(\omega),u(\omega),\omega)dP(\omega)\quad\text{over $x\in\N$},
\end{equation}
where, for given integers $n_t$ and $m$
\[
\N = \{(x_t)_{t=0}^T\,|\,x_t\in L^0(\Omega,\F_t,P;\reals^{n_t})\},
\]
$u\in L^0(\Omega,\F,P;\reals^m)$ is the parameter and $f$ is an extended real-valued $\B(\R^n)\otimes\B(\R^m)\otimes\F$-measurable function, where $n:=n_0+ \ldots +n_T$. Here and in what follows, we define the expectation of a measurable function $\phi$ as $+\infty$ unless the positive part $\phi^+$ is integrable\footnote{In particular, the sum of extended real numbers is defined as $+\infty$ if any of the terms equals $+\infty$.}. The function $Ef$ is thus well-defined extended real-valued function on $\N\times L^0(\Omega,\F,P;\reals^m)$. We will assume throughout that the function $f(\cdot,\cdot,\omega)$ is {\em proper}, {\em lower semicontinuous} and {\em convex} for every $\omega\in\Omega$.

It was shown in \cite{pen11c} that, when applied to \eqref{p}, the conjugate duality framework of Rockafellar~\cite{roc74} allows for a unified treatment of many well-known duality frameworks in operations research and mathematical finance. In that context, the absence of a duality gap is equivalent to the closedness of the optimal value function
\[
\varphi(u)=\inf_{x\in\N}Ef(x,u)
\]
over an appropriate space of measurable functions $u$. Pennanen and Perkki\"o~\cite{pp12} gave a simple algebraic condition on the integrand $f$ that guarantees that the optimum in \eqref{p} is always attained and $\varphi$ is closed; see Section~\ref{sec:lsc} below. The condition provides a far reaching generalization of the classical no-arbitrage condition in financial mathematics but it also assumes that $f$ is bounded from below. In the financial context, the lower bound excludes, e.g., portfolio optimization problems with utility functions that are not bounded from above. That such a bound is superfluous is suggested e.g.\ by the results of R\'asonyi and Stettner~\cite{rs5} where the existence of solutions to portfolio optimization in the classical perfectly liquid market was obtained for more general utility functions.

This article relaxes the boundedness assumption on $f$. A key result for this is Lemma~\ref{lem:1} which, in turn, is based on local martingale techniques that are well-known in mathematical finance. We also prove an expression for the recession function of the optimal value function in terms of $f$. This is of interest when, e.g., one studies robust no arbitrage and robust no scalable arbitrage properties and the related dominating markets; see \cite{krs2,pp10}. 

In Section~\ref{sec:MF} we apply the main results to an optimal investment problem in liquid markets. Under the assumptions that prices are bounded from below and that the initial prices are bounded, we recover an existence result by R\'asonyi and Stettner~\cite[Theorem 6.2]{rs5} who assumed a well-known asymptotic elasticity condition of the utility function. Moreover, here we obtain the closedness of the associated value function defined on an appropriate space of future liabilites that is important in valuation of contingent claims; see \cite{pen14b}. During the recent years, market models with proportional transaction costs or other nonlinear illiquidity effects have received an increasing interest \cite{cms14,jk95,pp10}. Extensions of our results to such models will be analyzed in a forthcoming article.

\section{Closedness of the value function}\label{sec:lsc}

In this section we will first establish the closedness of the value function associated with \eqref{p} in the space of measurable functions $L^0$ with respect to the convergence in measure. Using this, we will then establish the closedness on locally convex subspaces of $L^0$. Recall that a function is {\em closed} if it is lower semicontinuous and either proper or a constant. A function is {\em proper} if it never takes the value $-\infty$ and it is finite at some point.

An extended real-valued function $h$ on $\R^n\times\Omega$, for a complete probability space $\Omega$, is a \emph{normal integrand} if $h$ is jointly measurable and $h(\cdot,\omega)$ is lower semicontinuous for all $\omega$; see \cite[Corollary 14.34]{rw98}. Thus,  we may say that $f$ is an $\F$-measurable {\em proper convex normal integrand} on $\R^n\times\R^m$.

In all of the results, the statements concerning the recession functions are new. For a normal integrand, $f^\infty$ is defined $\omega$-wise as the {\em recession function} of $f(\cdot,\cdot,\omega)$, i.e.
\[
f^\infty(x,u,\omega) = \sup_{\alpha>0}\frac{f(\bar x+\alpha x,\bar u+\alpha u,\omega) - f(\bar x,\bar u)}{\alpha},
\]
which is independent of the choice $(\bar x,\bar u)\in\dom f(\cdot,\cdot,\omega)$. By \cite[Theorem~8.5]{roc70a}, the supremum equals the limit as $\alpha\to+\infty$.

\begin{theorem}\label{thm:pp12}
Assume that there exists $m\in L^1$ such that
\[
f(x,u,\omega)\ge m(\omega)
\]
for all $(x,u,\omega)\in\R^n\times\R^m\times\Omega$ and that
\[
\L=\{x\in\N\mid f^\infty(x,0) \le 0\}
\] 
is a linear space. Then
\[
\varphi(u)=\inf_{x\in\N} E f(x,u)
\]
is closed and proper in $L^0$ and the infimum is attained for every $u\in L^0$. Moreover,
\[
\varphi^\infty(u)=\inf_{x\in\N} Ef^\infty(x,u).
\]
\end{theorem}

\begin{proof}
Theorem 2 of \cite{pp12} gives closedness of $\varphi$ with respect to a locally convex topological space $\U\subset L^0$ but the main argument of the proof establishes closedness with respect to the convergence in measure provided that $f$ has an integrable lower bound.

Let $\bar u\in\dom\varphi$ and $\bar x\in \N$ be such that $\varphi(\bar u)=Ef(\bar x,\bar u)$. Such $\bar x$ exists by \cite[Theorem 2]{pp12}. We have that
\begin{align*}
\varphi^\infty(u) &= \sup_{\alpha>0}\frac{\varphi(\bar u + \alpha u)-\varphi(\bar u)}{\alpha} = \sup_{\alpha>0}\inf_{x\in \N}Ef_\alpha(x,u),
\end{align*}
where
\[
f_\alpha(x,u) = \frac{f(\bar x+\alpha x,\bar u + \alpha u)-f(\bar x,\bar u)}{\alpha}.
\]
We have
\[
\varphi^\infty(u) \le \inf_{x\in \N}\sup_{\alpha>0}Ef_\alpha(x,u) \le \inf_{x\in \N}E[\sup_{\alpha>0}f_\alpha(x,u)] = \inf_{x\in X} Ef^\infty(x,u).
\]
To prove the converse, let $a>\sup_{\alpha>0}\inf_{x\in\N}Ef_\alpha(x,u)$. For every positive integer $\alpha$, there is an $x^\alpha\in\N$ with $Ef_\alpha(x^\alpha,u)<a$. The functions $f_\alpha$ are non-decreasing in $\alpha$, so $Ef_1(x^\alpha,u)<a$ and we may proceed as in the proof of \cite[Theorem 2]{pp12} (where we may assume that $x_t^\alpha\in N_t^\perp$ for every $t$ by \cite[Lemma 2]{pp12}) to obtain a sequence of convex combinations $\tilde x^\alpha= \co\{x^{\alpha'} \mid \alpha'\ge\alpha\}$ such that $\tilde x^\alpha\to\bar x$ almost surely. By Fatou's Lemma,
\[
Ef^\infty(\bar x,u)\le a,
\]
which completes the proof.
\end{proof}
 
When extending the theorem above to objectives that do not have a uniform lower bound, a key role is played by the following lemma, where 
\[
\N^\perp = \{v\in L^1(\Omega,\F,\reals^n)\,|\,E(x\cdot v)=0\ \forall x\in\N^\infty\}
\]
is the \emph{annihilator} of $\N^\infty=L^\infty\cap\N$.

\begin{lemma}\label{lem:1}
Let $x\in\N$ and $v\in\N^\perp$. If $E[x\cdot v]^+\in L^1$, then $E(x\cdot v)=0$.
\end{lemma}

\begin{proof}
Any $\xi\in L^1$ can be represented as $\xi=\sum_{t'=0}^{T+1}\Delta \xi_{t'}$ where $\xi_{-1}=0$ and $(\xi_{t'})_{t'=0}^{T+1}$ is the martingale  defined as $\xi_{T+1}=\xi$ and $\xi_{t'}=E_{t'}\xi$. For $\xi^i=v_i$, we have $\xi^i_{t'}=0$ for all $t'\le i$, so $x_t\cdot v_t=m^i_{T'+1}$, where $m^i$ is a local martingale defined by
\[
m^i_{t}=\sum_{t'=0}^{t} x_i\cdot \Delta \xi^i_{t'}.
\]
Thus $x \cdot v =m_{T+1}$ for $m=\sum_{i=0}^{T} m^i$. Since $Em_{T+1}<\infty$, we have that $m$ is a martingale (\cite[Theorem 2]{js98}) and thus $E(x\cdot v)=E m_{T+1}= E m_0=0$.
\end{proof}

The following additivity property of the extended real-valued expectation will often be used without a mention.

\begin{lemma}\label{lem:2}
Let $\phi_1$ and $\phi_2$ be extended real-valued measurable functions. If either $\phi_1^+,\phi_2^+\in L^1$ or $\phi_2\in L^1$, then
\[
E(\phi_1+\phi_2) = E\phi_1+E\phi_2.
\]
\end{lemma}
% \begin{proof}
% Writing the identity
% \[
% (\phi_1+\phi_2)^+-(\phi_1+\phi_2)^- = \phi_1^+-\phi_1^-+\phi_2^+-\phi_2^-,
% \]
% as
% \[
% (\phi_1+\phi_2)^+ +\phi_1^-+\phi_2^-=(\phi_1+\phi_2)^- +\phi_1^++\phi_2^+,
% \]
% and integrating, gives
% \[
% E(\phi_1+\phi_2)^+ +E \phi_1^-+E \phi_2^-=E(\phi_1+\phi_2)^- +E \phi_1^++E \phi_2^+.
% \]
% Since $\phi_1^+-\phi_2^-\le(\phi_1+\phi_2)^+\le \phi_1^++\phi_2^+$, we get $(\phi_1+\phi_2)^+\in L^1$ and that $(\phi_1+\phi_2)^-\in L^1$ if and only if $\phi_1^-,\phi_2^-\in L^1$. Thus, $E(\phi_1+\phi_2) = E\phi_1 + E\phi_2$.
% \end{proof}

The main contribution of this paper is contained in the following result which relaxes the lower bound of $f$ with respect to $x$.

\begin{theorem}\label{thm:1}
Assume that there exist $\lambda>0$ and $(v,\beta)\in\N^\perp\times L^1$ such that
\[
f(x,u,\omega)\ge x\cdot v(\omega) + \lambda [x\cdot v(\omega)]^+ + \beta(\omega)
\]
for all $(x,u,\omega)\in\R^n\times\R^m\times\Omega$ and that
\[
\L=\{x\in\N\mid f^\infty(x,0) \le 0\}
\] 
is a linear space. Then
\[
\varphi(u)=\inf_{x\in\N} E f(x,u)
\]
is closed and proper in $L^0$ and the infimum is attained for every $u\in L^0$. Moreover,
\[
\varphi^\infty(u)=\inf_{x\in\N} Ef^\infty(x,u).
\]
\end{theorem}

\begin{proof}
Note first that $\L=\{x\in\N\mid f^\infty(x,0) - x\cdot v\le 0\}$. Indeed, the lower bound on $f$ implies $f^\infty(x,0,\omega)\ge x\cdot v(\omega) + \lambda[x\cdot v(\omega)]^+$, so if $x$ belongs either to $\L$ or $\{x\in\N\mid f^\infty(x,0) - x\cdot v\le 0\}$, then $[x\cdot v]^+\le 0$ and thus, $x\cdot v=0$, by Lemma~\ref{lem:1}. Applying Theorem~\ref{thm:pp12} to the normal integrand $f_v(x,u,\omega) := f(x,u,\omega) - x\cdot v(\omega)$ we see that
\[
\varphi_v(u)=\inf_{x\in\N} E f_v(x,u)
\]
is closed and proper, the infimum is attained for every $u\in L^0$ and that
\[
\varphi_v^\infty(u)=\inf_{x\in\N} E f_v^\infty(x,u).
\] 
If either $(x,u)\in\dom Ef$ or $(x,u)\in\dom Ef_v$, the lower bound implies $[x\cdot v]^+\in L^1$ so that $x\cdot v\in L^1$ and $E(x\cdot v)=0$, by Lemma~\ref{lem:1}. In both cases, Lemma~\ref{lem:2} gives
\[
Ef_v(x,u) = Ef(x,u)-E(x\cdot v) = Ef(x,u)
\]
so that $\varphi_v=\varphi$. Similarly, $Ef^\infty=Ef_v^\infty$ so that $\varphi^\infty(u)=\inf_{x\in\N} Ef^\infty(x,u)$.
\end{proof}

Theorem~\ref{thm:1} is not valid if, in the lower bound for $f$, we allow $\lambda$ to be zero.
\begin{example}
Let $n=1$ and let $\F_0$ be such that there exist $\alpha\in L^1(\F_0)$ with $\alpha\notin L^2$ and nonzero $\beta\in L^\infty$ independent of $\F_0$  with $E[\beta\mid \F_0]=0$. We define
\[
f(x,u,\omega)=\|x-\alpha(\omega)\|+x\cdot v(\omega),
\]
where $ v=\alpha\beta$ so that $v\in \N^\perp$ and $f(x)\ge x\cdot v$. The reader may verify that
\[
\inf_{x\in\N}Ef(x,u)=0
\]
and the only nontrivial candidate for $Ef(x,u)=0$ is $x=\alpha$, but $Ef(\alpha,u)=\infty$.
\end{example}

The rest of this section is concerned with closedness of the value function on locally convex subspaces of $L^0$. We will assume from now on that the parameter $u$ belongs to a decomposable space $\U\subset L^0$ which is in separating duality with another decomposable space $\Y\subset L^0$ under the bilinear form
\[
\langle u,y\rangle = E(u\cdot y).
\]
Recall that $\U$ is {\em decomposable} if
\[
\one_Au+\one_{\Omega\setminus A}u'\in\U
\]
whenever $A\in\F$, $u\in\U$ and $u'\in L^\infty$; see e.g.\ \cite{roc76}.

The families of closed convex sets coincide in the weak $\sigma(\U,\Y)$ and in the Mackey topologies $\tau(\U,\Y)$ \cite[p. 132]{sch71}, so we may say that $\varphi$ is closed in $\U$ whenever it is so with respect to either of the topologies. The closed convex function 
\[
\varphi^*(y)=\sup_{u\in \U}\{\langle u,y\rangle-\varphi(u)\}
\]
is called the \emph{conjugate} of $\varphi$. When $\varphi$ is closed, it has the {\em dual representation}
\begin{align*}
\varphi(u) = \sup_{y\in\Y}\{\langle u,y\rangle - \varphi^*(y)\};
\end{align*}
see \cite[Theorem 5]{roc74}. Dual representations are behind many fundamental results in mathematical finance as in the classical formula where superhedging prices of contingent claims in liquid markets are given in terms of martingale measures; for this and some recent developments in more general market models, see \cite{pen12a} and \cite{pen14a}. 

Using traditional topological arguments on decomposable spaces (see e.g.\ Rockafellar~\cite{roc76} or Ioffe~\cite{iof77}), we may relax in Theorem~\ref{thm:2} the lower bound on $f$ with respect to $u$ as well. The following theorem generalizes \cite[Theorem 2]{pp12} by relaxing the uniform lower boundedness of $f$ with respect to $x$. 

\begin{theorem}\label{thm:2}
Assume that there exist $\lambda>0$, $v\in\N^\perp$ and a convex normal integrand $g$ on $\reals^m\times\Omega$ such that $Eg$ is $\tau(\U,\Y)$-continuous, $Eg^\infty$ is finite on $\U$,
\[
f(x,u,\omega)\ge x\cdot v(\omega) + \lambda [x\cdot v(\omega)]^+ - g(u,\omega)
\]
for all $(x,u,\omega)\in\R^n\times\R^m\times\Omega$ and such that
\[
\L=\{x\in\N\mid f^\infty(x,0)\le 0\}
\] 
is a linear space. Then
\[
\varphi(u)=\inf_{x\in\N} E f(x,u)
\]
is closed and proper in $\U$ and the infimum is attained for every $u\in\U$. Moreover,
\[
\varphi^\infty(u)=\inf_{x\in\N} E f^\infty(x,u).
\]
\end{theorem}

\begin{proof}
Applying Theorem~\ref{thm:1} to the normal integrand $f_g(x,u,\omega)=f(x,u,\omega)+g(u,\omega)$, we get that
\[
\varphi_g(u) = \inf_{x\in\N}Ef_g(x,u)
\]
is proper and closed on $L^0$, that the infimum is attained for every $u\in L^0$ and that
\[
\varphi_g^\infty(u)=\inf_{x\in\N} E f_g^\infty(x,u).
\]
Closedness in $L^0$ implies closedness in the relative topology of $L^1$ which, by \cite[Lemma~6]{pp12}, implies that $\varphi_g$ is closed in $\U$. By Lemma~\ref{lem:2}, $E[f(x,u)+g(u)]=Ef(x,u)+Eg(u)$ for all $(x,u)\in\N\times\U$ and thus
\[
\varphi_g = \varphi+Eg
\]
on $\U$. The $\tau(\U,\Y)$-continuity of $Eg$ implies the $\tau(\U,\Y)$-closedness of $\varphi=\varphi_g-Eg$. 

As to the recession function, $f^\infty_g=f^\infty+g^\infty$ \cite[Theorem 9.3]{roc70a}, 
so
\[
\varphi_g^\infty(u)=\inf_{x\in\N} E f^\infty(x,u) + Eg^\infty(u).
\]
Similarly, closedness of $\varphi$ and $Eg$ imply 
\[
\varphi_g^\infty=\varphi^\infty+(Eg)^\infty,
\]
where, by the monotone convergence theorem, $(Eg)^\infty=Eg^\infty$.
\end{proof}

The following lemma gives a sufficient condition for the first hypothesis in Theorem~\ref{thm:2}.  The characterization involves the conjugate normal integrand $f^*:\reals^n\times\reals^m\times\Omega\to\ereals$ of $f$ defined $\omega$-wise as follows
\[
f^*(v,y,\omega) = \sup\{x\cdot v + u\cdot y - f(x,u,\omega)\}.
\]
By \cite[Theorem 14.50]{rw98}, $f^*$ is indeed a normal integrand.

\begin{lemma}\label{lem:dualtest}
Assume that there exists $v\in\N^\perp$ and $\lambda>0$ such that the function
\[
\gamma(v) = \inf_{y\in\Y} Ef^*(v,y)
\]
is finite at $v$ and $(1+\lambda)v$. Then
\[
f(x,u,\omega)\ge x\cdot v(\omega) + \lambda [x\cdot v(\omega)]^+ - g(u,\omega)
\]
for all $(x,u,\omega)\in\R^n\times\R^m\times\Omega$, where $g(u,\omega)=\max_i{u\cdot y^i(\omega)}+\beta(\omega)$ for some $y^i\in\Y$ and $\beta\in L^1$.
\end{lemma}

\begin{proof}
The assumption means that there exist $y^1,y^2\in\Y$ and $\beta^1,\beta^2\in L^1$ such that
\[
f^*(\lambda^i v,y^i) \le \beta^i
\]
where $\lambda^1=\lambda$ and $\lambda^2=1+\lambda$. Equivalently,
\[
f(x,u,\omega)\ge \lambda^i x\cdot v(\omega) + u\cdot y^i(\omega) - \beta^i(\omega)\quad i=1,2
\]
for all $(x,u,\omega)\in\R^n\times\R^m\times\Omega$, which implies
\[
f(x,u,\omega)\ge \max_i\{\lambda^i x\cdot v(\omega)\} + \min_i{u\cdot y^i(\omega)} - \min_i \beta^i(\omega)
\]
for all $(x,u,\omega)\in\R^n\times\R^m\times\Omega$.
\end{proof}

\section{Application to mathematical finance}\label{sec:MF}

We will consider optimal investment on a financial market with a finite set $J$ of assets from the point of view of an agent who has a financial liability described by a payment $u\in L^0(\F_T)$ to be made at the terminal time. As is usual, we express the terminal wealth as a stochastic integral with respect to the adapted price process $s$ so that the problem can be written as
\begin{equation}\label{alm}\tag{ALM}
\minimize\quad E  V\left(u-\sum_{t=0}^{T-1} x_t\cdot\Delta s_{t+1}\right)\quad\ovr\quad x\in\N_0,
\end{equation}
where $\N_0=\{x\in\N\mid x_T=0\ P\text{-a.s.}\}$, $x_{-1}=0$ and $V$ is a nondecreasing nonconstant convex function with $V(0)=0$. The objective of the agent is thus to find a trading strategy $x$ that hedges against the liability $u$ as well as possible as measured by expected ``disutility'' at terminal time. The change of signs $U(u)=-V(-u)$ transforms the problem into a usual maximization of expected terminal utility. There is vast literature on \eqref{alm} and on its extensions to more sophisticated models of optimal investment; see the references in \cite{pen14b}.

The following theorem gives conditions for the closedness in $\U$ of the value function $\varphi$ of \eqref{alm} and thus, the validity of the dual representation
\[
\varphi(u) = \sup_{y\in\Y}\{\langle u,y\rangle - \varphi^*(y)\}.
\]
Recall that $s$ satisfies the {\em no-arbitrage condition} if
\begin{equation}\tag{NA}\label{na}
\left\{\sum_{t=0}^{T-1}x_t\cdot\Delta s_{t+1}\midb \sum_{t=0}^{T-1}x_t\cdot\Delta s_{t+1}\ge 0,\ x\in\N \right\}=\{0\}.
\end{equation}  
\begin{theorem}\label{thm:alm}
Assume that we have \eqref{na} and that there exists a martingale measure $Q\ll P$ of $s$ such that $EV^*(\lambda\frac{dQ}{dP})<\infty$ for two different $\lambda\ge 0$. If $\frac{dQ}{dP}\in\Y$, then the value function of \eqref{alm} is closed in $\U$ and \eqref{alm} has a solution for all $u\in\U$.
\end{theorem}
\begin{proof}
We fit \eqref{alm} in the general model \eqref{p} with\footnote{Here $\delta_C$ denotes the {\em indicator function} of a set $C$, i.e.\ $\delta_C(x)=0$ if $x\in C$ and $\delta_C(x)= +\infty$ otherwise.}
\[
f(x,u,\omega)= V\left(u-\sum_{t=0}^{T-1} x_t\cdot\Delta s_t(\omega)\right)+\delta_{0}(x_T).
\]
We have
\[
f^\infty(x,0,\omega) = V^\infty\left(-\sum_{t=0}^{T-1}x_t\cdot\Delta s_{t+1}\right)+\delta_0(x_T).
\]
Since $V$ is nonconstant, we get $V^\infty(u)>0$ for $u>0$ and $V^\infty(u)\le 0$ for $u\le 0$, so the linearity condition in Theorem~\ref{thm:2} reduces to \eqref{na}. In order to apply Lemma~\ref{lem:dualtest}, we calculate
\begin{align*}
f^*(v,y,\omega) &= \sup_{x\in\reals^n,u\in\reals}\left\{ x\cdot v + uy-V\left(u-\sum_{t=0}^{T-1} x_t\cdot\Delta s_{t+1}(\omega)\right)-\delta_0(x_T)\right\}\\
&= V^*(y)+\sup_{x\in\reals^n}\left\{ \sum_{t=0}^{T-1}x_t\cdot (y\Delta s_{t+1}(\omega)+v_t)\right\}\\
&= V^*(y)+\sum_{t=0}^{T-1}\delta_{0}(y\Delta s_{t+1}(\omega)+v_t).
\end{align*}
We choose $y=\frac{dQ}{dP}$ and $v=-\frac{dQ}{dP}\Delta s_{t+1}$ so that
\begin{align*}
Ef^*(\lambda v,\lambda y) &= EV^*\left(\lambda \frac{dQ}{dP}\right).
\end{align*}
Here $Ef^*(\lambda v,\lambda y)$ is finite for two different $\lambda$ by assumption so that, by Lemma~\ref{lem:dualtest}, we may apply Theorem~\ref{thm:2}. 
\end{proof}

\begin{remark}
The proof of Theorem~\ref{thm:alm} also gives a formula for the recession function of the value function associated with \eqref{alm}. That is, since the value function is closed, its recession function is closed as well which, by the recession formula in Theorem~\ref{thm:2}, is equivalent with the fact that the set 
\[
\C = \{u\in\U\,|\, \exists x\in\N_0:\  \sum_{t=0}^{T-1} x_t\cdot\Delta s_{t+1} \ge u\}
\]
of claims that can be superhedged without a cost is closed in $\U$. However, this result follows already from \cite[Theorem 2]{pp12}.
\end{remark}

\begin{remark}\label{rem:2}
Under the assumptions of Theorem~\ref{thm:alm}, if one is merely interested in the existence of solutions of \eqref{alm}, we point out the following. We may define
\[
f(x,u,\omega)= V\left(c(\omega)-\sum_{t=0}^{T-1} x_t\cdot\Delta s_{t+1}(\omega)\right)+\delta_{0}(x_T)
\]
so that $f$ is independent of $u$ and the Fenchel inequality implies that
\[
f(x,u,\omega)\ge x\cdot v(\omega) +c(\omega)\lambda\frac{dQ}{dP}(\omega)-V^*(\lambda\frac{dQ}{dP}(\omega))
\]
for all $(x,u,\omega)\in\R^n\times\R^m\times\Omega$, where $v\in\N^\perp$ is defined by $v_t=-\lambda\frac{dQ}{dP}\Delta s_{t+1}$. Thus, by Theorem~\ref{thm:1},  \eqref{alm} has a solution for $u=c$ whenever $c$ is $Q$-integrable.
\end{remark}

\begin{remark}\label{rem:3}
In general, the price process does not admit a martingale measure for which the integrability condition in Theorem~\ref{thm:alm} is satisfied. For example, consider a one-period market model with a trivial $\sigma$-algebra $\F_0$, $P(\Delta s_1=1)=\frac{3}{4}$, $P(\Delta s_t=-1)=\frac{1}{4}$, and a disutility function specified by 
\begin{align*}
V'(u)=\sum_{n=1}^\infty \left[\one_{(-n-1,-n]}(u)(1+\frac{1}{n^2})+\one_{(n,n+1]}(u)(3-\frac{1}{n^2})\right].
\end{align*}
This is Example 7.3 in \cite{rs5} where it has been shown that the optimal value of \eqref{alm} for $u=0$ is finite but optimal solutions do not exist. The interested reader may verify that the unique martingale measure is given by $\frac{dQ}{dP}=\frac{2}{3}(\one_{\Delta s_1=1}+3\one_{\Delta s_1=-1})$ and that $V^*(y)=+\infty$ whenever $y\notin [1,3]$ so that $EV^*(\lambda\frac{dQ}{dP})=+\infty$ for every $\lambda\neq \frac{3}{2}$. 
\end{remark}

The ``two-$\lambda$-condition'' in the theorem is close in spirit to \cite[Assumption~4.2]{bc11} since it implies a similar $\lambda$-condition for all $\lambda$ sufficiently small \cite[Corollary~4.4]{bc11}. In the cited article the authors work in a continuous time setting, here two different $\lambda$ suffice. 

The following corollary says that the two-$\lambda$-condition is implied by a simpler integrability condition for utility functions that satisfy either of the well-known \emph{asymptotic elasticity} conditions
\begin{align}
V^*(\lambda y) &\le CV^*_t(y,\omega)\quad\forall y\in[0,\bar y]\text{ for some } C<+\infty,\ \bar y>0,\ \lambda\in(0,1),\label{rae1}\\
V^*(\lambda y) &\le CV^*_t(y,\omega)\quad\forall y\ge \bar y\text{ for some } C<+\infty,\ \bar y>0,\ \lambda>1.\label{rae2}
\end{align}
These conditions together with their equivalent formulations were introduced in \cite{ks99} and \cite{sch1}, respectively. 

\begin{corollary}\label{cor:alm}
Assume that we have \eqref{na}, $V$ satisfies \eqref{rae1} or \eqref{rae2}, and that there exists a martingale measure $Q\ll P$ of $s$ such that $EV^*(\lambda \frac{dQ}{dP})<\infty$ for some $\lambda\in\R_+$. If $\frac{dQ}{dP}\in\Y$, then the value function of \eqref{alm} is closed in $\U$ and \eqref{alm} has a solution for all $u\in\U$.
\end{corollary}
\begin{proof}
We use the notation from the proof of Theorem~\ref{thm:alm}. For simplicity, we assume that $\lambda=1$. If we have \eqref{rae1}, then
\begin{align*}
EV^*(\lambda' \frac{dQ}{dP}) &= E\one_{\{\frac{dQ}{dP}\le \bar y\}}V^*(\lambda' \frac{dQ}{dP}) + E\one_{\{\frac{dQ}{dP}> \bar y\}}V^*(\lambda' \frac{dQ}{dP})\\
&\le E\one_{\{\frac{dQ}{dP}\le \bar y\}}V^*(\lambda' \frac{dQ}{dP}) + E\one_{\{\frac{dQ}{dP}> \bar y\}}\max\{V^*(\frac{dQ}{dP}), V^*(\lambda' \bar y)\}\\
&\le E\one_{\{\frac{dQ}{dP}\le \bar y\}}CV^*(\frac{dQ}{dP}) + E\one_{\{\frac{dQ}{dP}> \bar y\}}\max\{V^*(\frac{dQ}{dP}),CV^*(\bar y)\},
\end{align*}
so $Ef^*(\lambda' v,\lambda' y)$ is finite for some $\lambda'<1$. Similary, if we have \eqref{rae2}, then $Ef^*(\lambda' v,\lambda' y)$ is finite for some $\lambda'>1$.
\end{proof}

The next theorem gives conditions for the existence of a martingale measure used in Corollary~\ref{cor:alm}.  The third condition in the theorem means that the optimal value of \eqref{alm} is finite for some positive initial endowment. We say that the price process is bounded from below if there is a constant $a\le 0$ such that $s^j_t\ge a$ $P$-almost surely for all $t$ and $j\in J$. In the proof, much like in that of \cite[Proposition 3.2]{ks99}, we approximate the disutility function by disutility functions that are bounded from below.
\begin{theorem}\label{thm:lb}
Assume that we have \eqref{na}, $V$ satisfies \eqref{rae1}, the optimal value of \eqref{alm} is finite for some $u=c$, where $c<0$ is a constant, the price process is bounded from below, and that $s_0\in L^\infty$. Then there exists a martingale measure $Q$ of $s$ with $EV^*(\lambda \frac{dQ}{dP})<\infty$ for some $\lambda\in\R_+$.
\end{theorem}
\begin{proof}
For positive integers $n$, we define functions $\tilde\varphi_n$ on $\U=L^\infty(\F_T)$ by
\[
\tilde\varphi_n(u)=\inf_{x\in\N^\infty}E V_n\left(\beta u-\sum_{t=0}^{T-1}x_t\cdot\Delta s_{t+1}\right),
\]
where $V_n=\max\{V,-n\}$ and $\beta=\max_t\{1,|s_t|\}$. For $\Y=L^1(\F_T)$, we get 
\begin{align*}
\tilde\varphi_n^*(y) &= \sup_{x\in\N^\infty}\sup_{u\in\U}E\left[uy-V_n\left(\beta u-\sum_{t=0}^{T-1}x_t\cdot\Delta s_{t+1}\right)\right]\\
&= \sup_{x\in\N^\infty}E\left[y/\beta\sum_{t=0}^{T-1}x_t\cdot\Delta s_{t+1}\right]+EV_n^*(y/\beta)\\
&=\delta_{\mathcal Q}(y/\beta)+EV_n^*(y/\beta),
\end{align*}
where $\mathcal Q$ is the set of positive multiples of martingale measure densities of $s$. Here the second equality follows from the interchange rule \cite[Theorem~14.60]{rw98}, since, for every $x\in\N^\infty$, there exists $u\in \U$ such that $EV_n(\beta u-\sum_{t=0}^{T-1}x_t\cdot\Delta s_{t+1})<\infty$. 

We have that $\varphi\le\varphi_n\le\tilde\varphi_n$, where 
\begin{align*}
\varphi(u) &=\inf_{x\in\N}E V\left(\beta u-\sum_{t=0}^{T-1}x_t\cdot\Delta s_{t+1}\right),\\
\varphi_n(u) &=\inf_{x\in\N}E V_n\left(\beta u-\sum_{t=0}^{T-1}x_t\cdot\Delta s_{t+1}\right).
\end{align*}
We get as in the proof of Theorem~\ref{thm:alm} that $\varphi_n$ are closed.\footnote{In fact, we may apply \cite[Theorem 2]{pp12}.} Therefore $\varphi(c/\beta)\le\varphi_n(c/\beta)\le\tilde\varphi_n^{**}(c/\beta)$ and, by the above expression for $\tilde\varphi^*_n$, there exist $y^n/\beta\in\Q$ such that 
\begin{align*}
\lim_n E[cy^n/\beta-V_n^*(y^n/\beta)+n^{-1}]\ge\lim_n \varphi_n(c/\beta)\ge \varphi(c/\beta)>-\infty.
\end{align*}
Since $V_n^*$ are nonnegative and $c$ is negative, $y^n/\beta$ are bounded in $L^1$. Thus Koml\'os' theorem (see e.g. \cite{ks9}) implies the existence of a sequence of convex combinations $\tilde y^n/\beta\in \co\{y^{n'}/\beta \mid n'\ge n\}$ such that $\tilde y^n/\beta$ converge to some $\bar y/\beta\in L^1$ $P$-almost surely. By Fatou's lemma and convexity,
\begin{align}\label{eq:it0}
E[c\bar y/\beta-V^*(\bar y/\beta)]\ge\limsup_n E[c\tilde y^n/\beta-V_n^*(\tilde y^n/\beta)]\ge\varphi(c/\beta)>-\infty.
\end{align}

Since $V$ satisfies \eqref{rae1}, we get, as in the proof of Corollary~\ref{cor:alm}, that $EV^*(2^{-\nu}\bar y/\beta)$ is finite. Since $V_n^*$ increase to $V^*$, there is, for every $\nu$, an $n(\nu)$ such that
\begin{align}\label{eq:it}
EV_{n}^*(2^{-\nu}\bar y/\beta)\ge EV^*(2^{-\nu}\bar y/\beta)-2^{-\nu}
\end{align}
for all $n\ge n(\nu)$. Each of the functions $W_\nu:=V_{n(\nu+1)}^*-V_{n(\nu)}^*$ is bounded, so, by a diagonalization argument, we may assume that 
\begin{align*}
EW_\nu(2^{-\nu}\tilde y^n/\beta)\le EW_\nu(2^{-\nu}\bar y/\beta)+2^{-\nu}
\end{align*}
for all $n\ge n(\nu)$. Since $W_\nu\le V^*-V^*_{n(\nu)}$ and since we have \eqref{eq:it}, we get that
\begin{align}\label{eq:it2}
EW_\nu(2^{-\nu}\tilde y^n/\beta)\le 2^{-\nu+1}
\end{align}
for all $n\ge n(\nu)$. We define 
\[
\tilde y =\sum_{\nu=0}^\infty 2^{-\nu}\tilde y^{n(\nu)}.
\]
Firstly, since $V^*_n$ increase to $V^*$, we see from \eqref{eq:it0} that $V^*_{n(0)}(\tilde y^{n(\nu)}/\beta)$ are bounded in $L^1$ and thus $EV_{n(0)}^*(\tilde y/\beta)\le \sum_{\nu=0}^\infty 2^{-\nu} EV^*_{n(0)}(\tilde y^{n(\nu)}/\beta)<\infty$, where the first inequality follows from the lower semicontinuity and Fatou's lemma. Secondly, it is not difficult to verify that $W_\nu$ are decreasing functions so that, by \eqref{eq:it2}, $EW_\nu(\tilde y/\beta) \le 2^{-\nu+1}$. 
%\begin{align*}
%EW_\nu(\tilde q/\beta) \le W_\nu (\sum_{\nu'=1}^\infty 2^{\nu+\nu'}\tilde q^{\nu+\nu'}/\beta) &\le \sum_{\nu'=1}^\infty 2^{-\nu'}E W_\nu( 2^{-\nu}\tilde q^{n(\nu)+\nu'}/\beta)\le 2^{-\nu+1}.
%EW_\nu(\tilde q/\beta) \le 2^{-\nu+1}.
%\end{align*}
Thus
\[
EV^*(\tilde y/\beta)=E V^*_{n(0)}(\tilde y/\beta)+\sum_{\nu=0}^\infty EW_\nu(\tilde y/\beta)<\infty.
\]
Since the price process is bounded from below, we get from the monotone convergence theorem that
\[
E\left[\frac{\tilde y}{\beta} \one_A s_{t+1}\right]=\sum_{\nu=1}^\infty 2^{-\nu}E\left[\frac{\tilde y^\nu}{\beta} \one_A s_{t+1}\right]=\sum_{\nu=1}^\infty 2^{-\nu}E\left[\frac{\tilde y^\nu}{\beta} \one_A s_{t}\right] = E\left[\frac{\tilde y}{\beta} \one_A s_t\right]
\]
 for every $t$ and $A\in\F_t$. Finally, by monotone convergence again, $E[\tilde y/\beta s_t]=E[\tilde y/\beta s_0]<\infty$, so $\tilde y/\beta$ is a positive multiple of a martingale measure density.

% Since, for every $x\in\N^\infty$, there is $c\in L^\infty$ such that 
% \[
% EV\left(\beta c-\sum_{t=0}^{T-1}x_t\cdot\Delta s_{t+1}\right)<\infty,
% \]
% the Fenchel inequality 
% \[
% \tilde q/\beta\left(\beta c-\sum_{t=0}^{T-1}x_t\cdot\Delta s_{t+1}\right)\le V^*(\tilde q/\beta)+V\left(\beta c-\sum_{t=0}^{T-1}x_t\cdot\Delta s_{t+1}\right)
% \]
% implies that

% For every positive integer $m$ and a stopping time $\tau_m=\inf\{t\mid |s_t|\ge m\}$, we have using the dominated convergence theorem that $E[\tilde q \one_A s_{(t+1)\wedge \tau_m}]=\sum_{\nu=1}^\infty 2^{-\nu}E[\tilde q^\nu \one_A s_{(t+1)\wedge \tau_m}]=\sum_{\nu=1}^\infty 2^{-\nu}E[\tilde q^\nu \one_A s_{t\wedge \tau_m}] = E[\tilde q \one_A s_t]$ for every $t$ and $A\in\F_t$. Moreover, $E[-\tilde qs_T]\le E[V(c'-s_T)+V^*(\tilde q)-c'\tilde q]<\infty$ which implies, by \cite[Theorem 2]{js98}, that $\tilde q$ is positive multiple of a martingale measure density.
\end{proof}

Under the mild assumptions that the price process is bounded from below and that $s_0\in L^\infty$, we recover the following result on the existence of solutions by R\'asonyi and Stettner in \cite[Theorem 2.7]{rs5}. 
\begin{corollary}
Assume that we have \eqref{na}, there exists $\tilde u<0$ and $0<\gamma<1$ such that
\[
V(\lambda u)\ge \lambda^\gamma V(u)\quad\forall u\le \tilde u,\ \lambda\ge 1,
\]
the functions $V_T(u,\omega)=V(u)$ and
\begin{align*}
V_t(u) &=\essinf_{x_t\in L^0(\F_t)} E[V_{t+1}(u+x_t\cdot\Delta s_{t+1})\mid \F_t]
\end{align*}
on $\R\times\Omega$ are well-defined and proper, and that $EV_0(u)>-\infty$ for all constants $u$. If the price process is bounded from below and $s_0\in L^\infty$, then \eqref{alm} has a solution for every $u\in L^\infty$.
\end{corollary}
\begin{proof}
By \cite[Corollary 6.1]{ks99}, \eqref{rae1} and the given growth condition for $V$ are equivalent. Thus, by Theorem~\ref{thm:lb},  Corollary~\ref{cor:alm} and Remark~\ref{rem:2}, it suffices to show that \eqref{alm} is finite for $u=c$, where $c<0$ is a constant.

Let $x\in\N$ be such that $EV(c-\sum_{t=0}^{T-1}x_t\cdot\Delta s_{t+1})<\infty$. The proofs of Proposition 4.2 and Proposition 4.4 in \cite{rs5} show that, outside an evanescent set,\footnote{A set in $\R\times\Omega$ is evanescent if its projection onto $\Omega$ is a null-set.} each  
\[
\tilde V_t(u,\omega)= E[V_{t+1}(u+x_t\cdot\Delta s_{t+1})\mid \F_t](\omega)
\]
and each $V_t(u,\omega)$ is convex, nondecreasing and lower semicontinuous in the $u$-argument and $\F_t$-measurable in the $\omega$-argument. Thus these functions are normal $\F_t$-integrands by \cite[Proposition 14.39]{rw98}, and it is therefore not difficult to verify that $V_t\le\tilde V_t$ outside an evanescent set. Consequently,
\[
V_t\left(c+\sum_{t'=0}^{t-1}x_{t'}\Delta s_{t'+1}\right)\le E\left[V_{t+1}\left(c+\sum_{t'=0}^{t}x_{t'}\Delta s_{t'+1}\right)\midb \F_t\right]\ P\text{-a.s.},
\]
where the sum on the left side is defined as zero for $t=0$. A repetition of these arguments for every $t$ gives
\[
EV_0(c)\le EV\left(c+\sum_{t=0}^{T-1}x_{t}\Delta s_{t+1}\right),
\]
from where we get the claim by taking the infimum over $x\in\N$.
 
\end{proof}

We finish by pointing out that our main results are applicable in much more general settings of optimal investment in illiquid markets as well; see \cite{pp12} for such example with utilities that are bounded from above. Extensions to other illiquid models together with unbounded utilities will be analyzed elsewhere.

\section*{Acknowledgments}
The author is grateful to the Einstein Foundation for the financial support. The author thanks Teemu Pennanen and Mikl\'os R\'asonyi for helpful comments and discussions.
\bibliographystyle{plain}
\bibliography{sp}

\end{document}